\documentclass[reqno,12pt]{amsart}
\usepackage{bbm}
\usepackage[all,pdf]{xy}
\usepackage{epsfig}
\usepackage{amsmath}
\usepackage{amssymb}
\usepackage{amscd}
\usepackage{graphicx}
\usepackage{bm}

\makeatletter
\@namedef{subjclassname@2020}{%
	\textup{2020} Mathematics Subject Classification}
\makeatother
\usepackage[colorlinks=true]{hyperref}
\usepackage{amsfonts}
\usepackage[top=25mm, bottom=27mm, left=27mm, right=27mm]{geometry}

\newtheorem{thm}{Theorem}[section]
\newtheorem{lemma}[thm]{Lemma}

\newtheorem{cl}{Claim}
\newtheorem{prop}[thm]{Proposition}

\newtheorem{ques}[thm]{Question}

\newtheorem{defn}[thm]{Definition}
\newtheorem{rem}[thm]{Remark}

\def \N {\mathbb N}

\def \Z {\mathbb Z}

\def \Q {\mathbb Q}

\numberwithin{equation}{section}

\begin{document}
		
	\baselineskip 14pt
	
	\title[]{monochromatic quotients, products and polynomial sums in the rationals}

	\author[]{Rongzhong Xiao}
	
	\address{School of Mathematical Sciences, University of Science and Technology of China, Hefei, Anhui, 230026, PR China}
	\email{xiaorz@mail.ustc.edu.cn}


 \begin{abstract}
	Let $k,a\in \N$ and let $p_1,\cdots,p_k\in \Q[n]$ with zero constant term. We show that for any finite coloring of $\Q$, there are non-zero $x,y\in \Q$ such that there exists a color which contains a set of the form $$\Big\{x,\frac{x}{y^a},x+p_{1}(y),\cdots,x+p_{k}(y)\Big\}$$ and there are non-zero $v,u\in \Q$ such that there exists a color which contains a set of the form $$\Big\{v,v\cdot {u^a},v+p_{1}(u),\cdots,v+p_{k}(u)\Big\}.$$
	\end{abstract}
		\maketitle
		
	\section{Introduction}
	In the investigation of partition of sets, looking for Ramsey family on $\N$ and $\Q$ is a center topic. A \textbf{Ramsey family} \bm{$\mathcal{A}$} \textbf{on} \bm{$\N$} is a finite set of the maps from ${\N}^{i}$ to $\Z$ where $i\in \N$ such that for any finite coloring of $\N$, there exists $x\in \N^{i}$ such that $\{f(x):f\in \mathcal{A}\}$ is monochromatic. Similarily, a \textbf{Ramsey family} \bm{$\mathcal{A}$} \textbf{on} $\Q$ is a finite set of the maps from ${\Q}^{i}$ to $\Q$ where $i\in \N$ such that for any finite coloring of $\Q$, there exists $x\in \Q^{i}$ such that $\{f(x):f\in \mathcal{A}\}$ is monochromatic.
	
	Naturally, we seek to search Ramsey family in $\Z[x_1,\cdots,x_s]$ and $\Q[x_1,\cdots,x_s]$ where $s\in \N$. 
	
	On $\N$, there are some results. I. Schur's theorem \cite{S} states the family $\{(x,y)\mapsto x,(x,y)\mapsto y,(x,y)\mapsto x+y\}$ is Ramsey on $\N$ and van der Waerden's thereom \cite{V} states for any $k\in \N$, the family $\{(x,y)\mapsto x,(x,y)\mapsto x+y,\cdots,(x,y)\mapsto x+ky\}$ is Ramsey on $\N$. For general linear polynomials, R. Rado built a equivalent condition for a family of linear polynomials to be Ramsey on $\N$ in \cite{R}. Based on it, we can verify that the family $\{x\mapsto x, x\mapsto x+3\}$ is not Ramsey on $\N$. For general polynomials, there are only a few results.  Furstenberg-Sark{\"o}zy theorem illustrates the family $\{(x,y)\mapsto x,(x,y)\mapsto x+y^{2}\}$ is Ramey on $\N$(see \cite{SA}) and V. Bergelson extended it to $\{(x,y)\mapsto x,(x,y)\mapsto y,(x,y)\mapsto x+y^{2}\}$ in \cite{B}. V. Bergelson and A. Leibman's polynomial extension of van der Waerden's thereom \cite{BL96} declares that for any $k\in \N$, for any $p_{1},\cdots,p_k\in \Z[n]$ with zero constant term, the family $\{(x,y)\mapsto x,(x,y)\mapsto x+p_{1}(y),\cdots,(x,y)\mapsto x+p_{k}(y)\}$ is Ramsey on $\N$. For aforementioned Ramsey familes on general polynomials, they do not contain polynomials $(x,y)\mapsto y$ and $(x,y)\mapsto x\cdot y$. For this, there exists a question which still lacks a complete answer .
	\begin{ques}\label{qu1}
		$($\cite[Question 3]{HS}$)$Is the family $\{(x,y)\mapsto x,(x,y)\mapsto y,(x,y)\mapsto x\cdot y,(x,y)\mapsto x+y\}$ Ramsey on $\N$? 
	\end{ques}
	For the question, J. Moreira answered it under leaving out polynomial $(x,y)\mapsto y$ in \cite[Corollary 1.5]{M}. If we consider the question on $\Q$, the fact that $(\Q\backslash \{0\},\cdot)$ is a group makes it easier than one in $\N$. Recently, M. Bowen and M. Sabok \cite[Theorem 1.1]{BS} showed that for any $k\in \N$, the family $\{(x,y)\mapsto x,(x,y)\mapsto y,(x,y)\mapsto x\cdot y,(x,y)\mapsto x+y,\cdots,(x,y)\mapsto x+ky\}$ is Ramsey on $\Q$. For general polynomials, J. Moreira 's theorem \cite[Theorem 1.4]{M} guarantees that for any $k\in \N$, for any $p_{1},\cdots,p_k\in \Z[n]$ with zero constant term, the family $\{(x,y)\mapsto x,(x,y)\mapsto x\cdot y,(x,y)\mapsto x+p_{1}(y),\cdots,(x,y)\mapsto x+p_{k}(y)\}$ is Ramsey on $\N$. Clearly, it is Ramsey on $\Q$. Our main result is to extend J. Moreira's family to a wider case on $\Q$ and it reflects the symmetry between multiplication and division on $\Q$. Specific statements are as follows.
	\begin{thm}\label{T1}
		Let $k,a\in \N$ and let $p_1,\cdots,p_k\in \Q[t]$ with zero constant term. For any finite coloring of $\Q$, then
		\begin{itemize}
			\item[(1)] there are non-zero $y\in \Q$ and an infinite subset $A$ of $\Q\backslash \{0\}$ such that $$A\cup (y^{-a}\cdot A) \cup (A+\{p_{i}(y):1\le i\le k\})$$ is monochromatic;
			\item[(2)] there are non-zero $u\in \Q$ and an infinite subset $B$ of $\Q\backslash \{0\}$ such that $$B\cup (u^{a}\cdot B) \cup (B+\{p_{i}(u):1\le i\le k\})$$ is monochromatic.
		\end{itemize}
	\end{thm}
    The proof of Theorem \ref{T1} is based on three ingredients. The first ingredient(see Theorem \ref{thm2}) is the multiple recurrence for polynomial mapping, built by V. Bergelson and A. Leibman in \cite{BL}, which helps us to build a van der Waerden-type result for piecewise syndetic subsets of $(\Q,+)$. The second ingredient(see Lemma \ref{lem1}), established by M. Bowen and M. Sabok in \cite{BS}, seeks to localize multiplicatively thick subsets according to the certain finite coloring of $\Q$. The third ingredient(see Theorem \ref{thm1}) is the partition regularity of piecewise syndetic subsets of $(\Q,+)$. 
    
    The organization of the paper is as follows. In section 2, we recall some large subsets and multiple recurrence for polynomial mappings and bulid a van der Waerden-type result. In section 3, we prove the Theorem \ref{T1}.
	\section{Preliminaries}
	\subsection{Some large subsets}
	At first, we state the definitions of some large subsets. Before this, we introduce some notations. Let $S$ be a non-empty set. Let $\mathcal{F}(S)$ denote all finite subsets of $S$ and $\mathcal{F}^{*}(S)$ denote all finite non-empty subsets of $S$.
	\begin{defn}
		Let $G$ be an infinite, countable, abelian group and let $A\subset G$. 
		\begin{itemize}
			\item[(a)] $A$ is \textbf{thick} if and only if for any $F\in \mathcal{F}^{*}(G)$, there is some $x\in G$ such that $Fx\subset A$; 
			\item[(b)] $A$ is \textbf{syndetic} if and only if there exists $F\in \mathcal{F}^{*}(G)$ such that $FA=G$;
			\item[(c)] $A$ is \textbf{piecewise syndetic} if and only if there exists $F\in \mathcal{F}^{*}(G)$ such that $FA$ is thick;
			\item[(d)] let $r\in\N$, $A$ is \bm{$IP_{r}$} if and only if there exist $s_1,\cdots,s_r\in G$ such that $FP(\{s_i\}_{i=1}^{r})=\{\prod_{i\in \alpha}s_i:\alpha\in \mathcal{F}^{*}(\{1,\cdots,r\})\}\subset A$;
			\item[(e)] let $r\in\N$, $A$ is \bm{$IP_{r}^{*}$} if and only if $A$ has non-empty intersection with any IP$_{r}$ subset of $G$.
		\end{itemize}
	\end{defn}
	The following results state some properties of the above large subsets.
	\begin{thm}\label{thm1}
		$($\cite[Theorem 2.5]{BLM}$)$Let $G$ be an infinite, countable, abelian group and let $A,B\subset G$. If $A\cup B$ is piecewise syndetic, then $A$ or $B$ is piecewise syndetic.
	\end{thm}
	\begin{prop}\label{prop1}
		Let $G$ be an infinite, countable, abelian group and let $A\subset G$. If $A$ is thick, then $A\backslash F$ is thick where $F\in \mathcal{F}^{*}(G)$.
	\end{prop}
	\begin{proof}
		For any $H\in \mathcal{F}^{*}(G)$, $\{x\in G:Hx\cap F\neq \varnothing\}$ is finite or empty. Note that $\{x\in G:Hx\subset A\}$ is thick. Then we can find $x\in G$ such that $Hx\subset A\backslash F$. So $A\backslash F$ is still thick. This finishes the proof.
	\end{proof}
	Next, we focus on specific group $\Q$. To avoid ambiguity, the thick subset of group $(\Q,+)$ is called \textbf{additively thick} and the thick subset of group $(\Q\backslash \{0\},\cdot)$ is called \textbf{multiplicatively thick}. 
	
	The following lemma provided in \cite[Lemma 3.3]{BS} which plays a crucial role in the proof of our main result.
	\begin{lemma}\label{lem1}
		Let $\Q\backslash \{0\}=\bigcup_{i=1}^{n}C_i$ be a finite coloring. There exist $k\in \N$, index sets $Y_1,\cdots,Y_k\subset \{1,\cdots,n\}$ and $F\in \mathcal{F}^{*}(\Q\backslash \{0\})$ such that 
		\begin{itemize}
			\item[(a)] for any $1\le l\le k$, $\bigcup_{m\in Y_l}C_m$ is multiplicatively thick;
			\item[(b)] for any $x\in \Q\backslash \{0\}$, there exists $1\le l\le k$ such that for each $m\in Y_l$, one has $x\in F\cdot C_m$.
		\end{itemize}
	\end{lemma}
	\subsection{Polynomial mapping $\mathcal{F}(S)\rightarrow G$ and multiple recurrence}
	Let $S$ be a non-empty set. Let $G$ be an infinite, countable, torsion-free abelian group with identity element $e_{G}$. We reproduce the notation of polynomial mapping $\mathcal{F}(S)\rightarrow G$ which introduced by V. Bergelson and A. Leibman in \cite[Section 1]{BL}.
	
	Let $\{g_t\}_{t\in T}$ be a collection of elements of $G$ indexed by a finite set $T$. We can define $\prod_{t\in T}g_t$. If $T$ is empty, we put $\prod_{t\in T}g_t=e_{G}$. Let $d\in \N$. We use $S^d$ to denote the produce $S\times \cdots \times S(d\  times)$. Conventionally, we let $S^{0}=\{\varnothing\}$. 
	
	Let $d\in \N\cup \{0\}$. A \textbf{monomial of degree} \bm{$d$} \textbf{on} \bm{$S$} \textbf{with values in} \bm{$G$} is a mapping $u:S^d\rightarrow G$. A monomial $u$ induces a \textbf{monomial mapping} $p_{u}:\mathcal{F}(S)\rightarrow G, \alpha \mapsto \prod_{s\in \alpha^{d}}u(s)$.
	
	A \textbf{polynomial mapping} $p:\mathcal{F}(S)\rightarrow G$ is the finite product of monomial mappings. The degree of $p$(denoted by $\deg p$) is the minimum, taken over the set of all representations of $p$ as the product $p=\prod_{i=1}^{m}p_{u_i}$ of monomial mappings, of the maximum of the degree of monomial $u_i,1\le i\le m$.
	
	The following theorem states the multiple recurrence phenomena for such polynomials.
	\begin{thm}\label{thm2}
		 $($\cite[Theorem 4.1]{BL}$)$Let $G$ be an infinite, countable, torsion-free abelian group of automorphisms of a compact metric space $(X,\rho)$. For any $k,d\in \N$ and any $\epsilon >0$, there exists $N\in \N$ such that if $S$ is a set with cardinality $\ge N$ and $p_1,\cdots,p_k:\mathcal{F}(S)\rightarrow G$ are polynomial mappings with $\deg p_i\le d,p_i(\varnothing)=Id_{X},1\le i\le k$, then there exist $x\in X$ and $\alpha\in \mathcal{F}^{*}(S)$ such that for each $1\le i\le k$, $\rho(x,p_{i}(\alpha)x)<\epsilon$.
	\end{thm}
	\begin{rem}
		Assume that X is \textbf{minimal} with respect to the action of $G$, that is, $X$ does not contain proper non-empty closed $G$-invariant proper subsets. By \cite[Proof of Theorem 4.1]{BL}, the set of the points $x\in X$ satisfy the requirements of the theorem is dense in $X$.
	\end{rem}
	\subsection{A van der Waerden-type result}
	Based on the Theorem \ref{thm2}, we have the following result.
	\begin{prop}\label{prop2}
		Let $G$ be an infinite, countable, torsion-free abelian group with identity element $e_{G}$ and A be a piecewise syndetic subset of $G$. For any $k,d\in \N$, there exists $N\in \N$ such that if $S$ is a set with cardinality $\ge N$ and $p_1,\cdots,p_k:\mathcal{F}(S)\rightarrow G$ are polynomial mappings with $\deg p_i\le d,p_i(\varnothing)=e_{G},1\le i\le k$, then there exists $\alpha\in \mathcal{F}^{*}(S)$ such that $A\cap p_{1}(\alpha)^{-1}A\cap \cdots \cap p_{k}(\alpha)^{-1}A$ is piecewise syndetic.
	\end{prop}
	\begin{proof}
		Let $\Omega=\{0,1\}^{G}$. The element of $\Omega$ can be written as $\textbf{w}=(w(g))_{g\in G}$. Specially, let $\textbf{0}$ denote element \textbf{y} with $y(g)=0$ for any $g\in G$. Since $G$ is countable, we can write $G$ as $\{g_1,g_2,\cdots,g_n,\cdots\}$. Define a metric $\rho$ on $\Omega$ by $$\rho(\textbf{w},\textbf{u})=\frac{1}{\min\{i\in \N:w(g_i)\neq u(g_i)\}}$$ for any $\textbf{w},\textbf{u}\in \Omega$. Then $(\Omega,\rho)$ is a compact metric space. $G$ can act on $(\Omega,\rho)$ by $(g\textbf{w})(h)=\textbf{w}(gh)$ for any $\textbf{w}\in \Omega,g,h\in G$.
		
		Define $\textbf{v}\in \Omega$ by $v(g)=1_{A}(g)$ for any $g\in G$. Let $X=\overline{\{g\textbf{v}:g\in G\}}$. Then we have the following claim. Its proof will be provided in the last paragraph.
		\begin{cl}\label{cl1}
			There exists $\textbf{0}\neq\textbf{x}\in X$ such that $Y$ is minimal with respect to the action of $G$ where $Y=\overline{\{g\textbf{x}:g\in G\}}$. 
		\end{cl}
		 Based on this claim, we can reach the conclusion. Let $U=\{\textbf{w}\in \Omega:w(e_{G})=1\}$. Apply Theorem \ref{thm2} to $(Y,\rho),k,d$ and $\frac{1}{2^{m}}$ where $g_m=e_{G}$, then there exists $N\in \N$ such that if $S$ is a set with cardinality $\ge N$ and $p_1,\cdots,p_k:\mathcal{F}(S)\rightarrow G$ are polynomial mappings with $\deg p_i\le d,p_i(\varnothing)=Id_{Y},1\le i\le k$, then there exist $\textbf{z}\in Y\cap U$ and $\alpha\in \mathcal{F}^{*}(S)$ such that for each $1\le i\le k$, $\textbf{z}(e_{G})=\textbf{z}(p_{i}(\alpha))$.
		 
		 Let $V=\{\textbf{w}\in \Omega:w(e_{G})=w(p_{1}(\alpha))=\cdots =w(p_{k}(\alpha))=1\}$. Since $(Y,G)$ is minimal, the set $B=\{g\in G:g\textbf{z}\in Y\cap V\}=\{a_1,a_2,\cdots,a_n,\cdots\}$ is a syndetic subset of $G$. Moreover, we can find a pairwise distinct sequence $\{h_j\}_{j\ge 1}\subset G$ such that for each $j\ge 1$, we have $h_{j}a_{1}\textbf{v},\cdots,h_{j}a_{j}\textbf{v}\in V$. That is, 
		 $$A\cap p_{1}(\alpha)^{-1}A\cap \cdots \cap p_{k}(\alpha)^{-1}A\supset \bigcup_{j\ge 1}\{h_{j}a_{1},\cdots,h_{j}a_{j}\}.$$ Therefore, $A\cap p_{1}(\alpha)^{-1}A\cap \cdots \cap p_{k}(\alpha)^{-1}A$ is piecewise syndetic.
		 
		 The rest of the proof is to verify Claim \ref{cl1}. Let $\{F_n\}_{n\ge 1}$ be a strictly increasing sequence of $\mathcal{F}^{*}(G)$ with $\bigcup_{n\ge 1}F_n=G$. Since $A$ is piecewise syndetic, there exist $H\in \mathcal{F}^{*}(G)$ and sequence $\{a_n\}_{n\ge 1}\subset G$ such that $a_{n}^{-1}F_{n}\subset H^{-1}A$. Without loss of generality, we can say that $a_n\textbf{v}\rightarrow \textbf{u}$ as $n\rightarrow \infty$ where $\textbf{u}\in X$. Let $C=\{g\in G:\textbf{u}(g)=1\}$. For any $h\in G$, there exist infinite $n\ge 1$ and $g\in H$ such that $gh\in a_{n}A$. That is, $gh\in C$. Clearly, $H^{-1}C=G$. So $C$ is syndetic. Let $Z=\overline{\{g\textbf{u}:g\in G\}}\subset X$. Clearly, $\textbf{0}\notin Z$. Take a non-empty closed $G$-invariant subset $Z'$ of $Z$ such that $(Z',G)$ is minimal. Then any element of $Z'$ can satisfy the requirments of Claim \ref{cl1}. This finishes the proof.
	\end{proof}
	Based on the above, we can bulid a van der Waerden-type result for piecewise syndetic subsets of infinite, countable, torsion-free abelian group. Before specific statements, we introduce the degree for polynomials between general abelian groups.
	\begin{defn}
		$($\cite[Definition 7.7]{BG}$)$Let $G$ and $H$ be abelian groups. Given $d\in \N$, a map $p:G\rightarrow H$ is a \textbf{polynomial of degree} \bm{$d$}  if the application of any $d+1$ of the discrete difference operators $\delta_g,g\in G$ defined by $(\delta_{g}p)(x)=p(gx)(p(x))^{-1}$ for any $x\in G$, reduces $p$ to the constant map which takes identity element of $H$.
	\end{defn}
	\begin{prop}\label{prop3}
		Let $H$ be an infinite, countable, torsion-free abelian group with with identity element $e_{H}$ and A be a piecewise syndetic subset of $H$. For any $k,d\in \N$, there exists $r\in \N$ such that for any infinite, countable, torsion-free abelian group $G$ with identity element $e_{G}$ and all polynomials $p_1,\cdots,p_k:G\rightarrow H$ of degree at most $d$ with $p_i(e_{G})=e_{H},1\le i\le k$, the set $$\{g\in G:A\cap p_{1}(g)^{-1}A\cap \cdots \cap p_{k}(g)^{-1}A\  is\ a\ piecewise\ syndetic\}$$ is an $IP_{r}^{*}$ subset of $G$.
	\end{prop}
	\begin{proof}
		By Proposition \ref{prop2}, we can get $r\in \N$ such that if $S$ is a set with cardinality $\ge r$ and $q_1,\cdots,q_k:\mathcal{F}(S)\rightarrow H$ are polynomial mappings with $\deg q_i\le d,q_i(\varnothing)=e_{H},1\le i\le k$, then there exists $\alpha\in \mathcal{F}^{*}(S)$ such that $A\cap q_{1}(\alpha)^{-1}A\cap \cdots \cap q_{k}(\alpha)^{-1}A$ is a piecewise syndetic subset of $H$.
		
		Choose $g_1,\cdots,g_r$ from $G$ arbitrarily. For any $1\le i\le k$, define polynomial mapping $\overline{p_i}:\mathcal{F}(\{1,\cdots,r\})\rightarrow H$ by the rule $\overline{p_i}(\alpha)=p_i(\prod_{m\in \alpha}g_i)$ for any $\alpha\in \mathcal{F}(\{1,\cdots,r\})$. Clearly, $\deg \overline{p_i}\le d,\overline{p_i}(\varnothing)=e_{H}$ for any $1\le i\le k$. So there exists $\beta\in \mathcal{F}^{*}(\{1,\cdots,r\})$ such that $A\cap \overline{p_1}(\beta)^{-1}A\cap \cdots \cap \overline{p_1}(\beta)^{-1}A$ is a piecewise syndetic subset of $H$. This finishes the proof.
	\end{proof}
	\section{Proof of Theorem \ref{T1}}
	In this section, we prove our main result. Here, we only provide proof for $(1)$ of Theorem \ref{T1} since the proof of the rest part is similar. During the process, the key points are Lemma \ref{lem1}, Proposition \ref{prop3} and pigeonhole principle.
	\begin{proof}[Proof of Theorem \ref{T1}]
		There exists $d\in \N$ such that $\deg p_i\le d,1\le i\le k$. Choose a finite coloring of $\Q$ arbitrarily and fix it. Then $\Q\backslash \{0\}$ can inherit a coloring from $\Q$. We write it as $\Q\backslash \{0\}=\bigcup_{m=1}^{n}C_{n}$. Clearly, $0$ has new color $n+1$ or there exists $\omega\in \{1,\cdots,n\}$ such that $0$ has color $\omega$.
		
		By Lemma \ref{lem1}, There exist $M\in \N$, index sets $Y_1,\cdots,Y_M\subset \{1,\cdots,n\}$ and $H\in \mathcal{F}^{*}(\Q\backslash \{0\})$ such that 
		\begin{itemize}
			\item[(1)] for any $1\le l\le M$, $\bigcup_{m\in Y_l}C_m$ is multiplicatively thick;
			\item[(2)] for any $x\in \Q\backslash \{0\}$, there exists $1\le l\le M$ such that for each $m\in Y_l$, one has $x\in H\cdot C_m$.
		\end{itemize}
		Let $s$ be a non-zero rational less than minimum of $H$. Let $F=H\cup \{s\}$. Then for any $x\in \Q\backslash \{0\}$, we can find minimal $1\le l_x\le M$ such that for each $m\in Y_{l_x}$, one has $x\in f_{m,x}\cdot C_m$ where $f_{m,x}=\min\{f\in H:x\in f\cdot C_m\}$. If $m\in \{1,\cdots,n\}\backslash Y_{l_x}$, let $f_{m,x}$ be $s$. Then we define new finite coloring of $\Q\backslash \{0\}$. That is, for any $x\in \Q\backslash \{0\}$, it has color $(l_x,f_{1,x},\cdots,f_{n,x})\in \{1,\cdots,M\}\times F^{n}$.
		
		By Theorem \ref{thm1} and Proposition \ref{prop1}, we know there exists $(l_1,f_{1,1},\cdots,f_{n,1})\in \{1,\cdots,M\}\times F^{n}$ such that the set $$A_1=\{x\in \Q\backslash \{0\}: x\ has\ color\ (l_1,f_{1,1},\cdots,f_{n,1})\}$$ is a piecewise syndetic subset of $(\Q,+)$. Let $N=36^{100M|F|^{n}},T=36^{100N|F|}k$. Apply Proposition \ref{prop3} to $T,d,A_1$, then we get a natural number $r_1$. We can construct $IP_{r_1}$ subset $S_{1}$ of $(\Q,+)$ such that $S_{1}\subset \bigcup_{m\in Y_{l_1}}C_m$. Let $$Q_1=\{f\cdot p_{i}(t):1\le i\le k,f\in F\}.$$ Clearly, $|Q_1|<T$. Then there exists $y_1\in S_{1}$ such that $$\tilde{A_1}=A_{1}\bigcap_{q\in Q_1}(A_{1}-q(y_1))$$ is a piecewise syndetic subset of $(\Q,+)$. 
		
		Next, we construct $r_j,A_j,\tilde{A_j},Q_j,y_j,S_{j},(l_j,f_{1,j},\cdots,f_{n,j})$ by induction until $j=N$ under the following requirements: for any $1\le j\le N$, we have 
		\begin{itemize}
			\item[(a)] $(l_j,f_{1,j},\cdots,f_{n,j})\in \{1,\cdots,M\}\times F^{n},r_j\in \N$;
			\item[(b)] $S_j$ is an $IP_{r_j}$ subset of $(\Q,+)$ and $S_{j}\subset \bigcup_{m\in Y_{l_j}}C_m$;
			\item[(c)] $y_j\in S_j$;
			\item[(d)] $Q_j=\Big\{(y_{1}\cdots y_{c-1})^{a}\cdot f\cdot p_{i}(t\cdot y_{c}\cdots y_{j-1}):f\in F,1\le i\le k, 1\le c<j\Big\}$ where we put $y_{1}\cdots y_{0}=1$ and $|Q_j|<T$;
			\item[(e)] $A_j$ and $\tilde{A_j}=A_{j}\bigcap_{q\in Q_j}(A_{j}-q(y_j))$ are two piecewise syndetic subsets of $(\Q,+)$.
		\end{itemize}
		and for any $1\le j< N$, we have
		\begin{itemize}
			\item[(f)] $A_{j+1}\subset A_{j}\bigcap_{q\in Q_j}(A_{j}-q(y_j))$;
			\item[(g)] $A_{j+1}=\{x\in \tilde{A_j}: x\cdot (\prod_{b=1}^{j}{y_b})^{-a}\ has\ color\ (l_{j+1},f_{1,j+1},\cdots,f_{n,j+1})\}$.
		\end{itemize}
		Clearly, we have finished construction for $j=1$. Let $j\ge 1$ and assume that $r_j,A_j,\tilde{A_j},Q_j,y_j,S_{j},(l_j,f_{1,j},\cdots,f_{n,j})$ have been constructed. 
		
		By Theorem \ref{thm1}, there exist a subset $A_{j+1}$ of $\tilde{A_j}$ which is a piecewise syndetic subset of $(\Q,+)$ and $(l_{j+1},f_{1,j+1},\cdots,f_{n,j+1})\in \{1,\cdots,M\}\times F^{n}$ such that $$A_{j+1}=\{x\in \tilde{A_j}: x\cdot (\prod_{b=1}^{j}{y_b})^{-a}\ has\ color\ (l_{j+1},f_{1,j+1},\cdots,f_{n,j+1})\}.$$ Let $$Q_{j+1}=\Big\{(y_{1}\cdots y_{c-1})^{a}\cdot f\cdot p_{i}(t\cdot y_{c}\cdots y_{j}):f\in F,1\le i\le k, 1\le c<j+1\Big\}. $$ Clearly, $|Q_{j+1}|<T$. Apply Proposition \ref{prop3} to $T,d,A_{j+1}$, then we get a natural number $r_{j+1}$. We can construct $IP_{r_{j+1}}$ subset $S_{j+1}$ of $(\Q,+)$ such that $S_{j+1}\subset \bigcup_{m\in Y_{l_{j+1}}}C_m$. And there exists $y_{j+1}\in S_{j+1}$ such that $$\tilde{A}_{j+1}=A_{j+1}\bigcap_{q\in Q_{j+1}}(A_{2}-q(y_{j+1}))$$ is a piecewise syndetic subset of $(\Q,+)$. 
		
		Obviously, there exist $2<\eta<j<N$ and $(l,f_{1},\cdots,f_{n})\in \{1,\cdots,M\}\times F^{n}$ such that $$(l,f_{1},\cdots,f_{n})=(l_j,f_{1,j},\cdots,f_{n,j})=(l_{\eta},f_{1,\eta},\cdots,f_{n,\eta})$$ and $j-\eta>2$. Let $y=y_{\eta}\cdots y_{j-1}$. Let $x'\in {A_j}$ and set $x=(f_m)^{-1}\cdot x' \cdot (y_{1}\cdots y_{\eta-1})^{-a}$ where $m\in Y_{l}$. So $x,\frac{x}{y^{a}}\in C_m$. Moreover, for any $q\in Q_{j-1}$, $x'+q(y_{j-1})\in A_{\eta}$. Then for any $q\in Q_{j-1}$, we have 
		$$x'\cdot (y_{1}\cdots y_{\eta-1})^{-a}+q(y_{j-1})\cdot (y_{1}\cdots y_{\eta-1})^{-a}\in f_{m}\cdot C_{m}.$$ Therefore, for each $1\le i\le k$, we have $x+p_{i}(y)\in C_m$ by definition of $Q_{j-1}$. This finishes the proof.
	\end{proof}
	In the above proof, we can not determine the color of $y$.
	For linear polynomials with zero constant term, V. Bergelson and D. Glasscock gave the upper Banach density version of Proposition \ref{prop3}(see \cite[Theorem 7.5]{BG}). By combining \cite[Proof of Theorem 4.3]{BS} and the above proof, we have the following result.
	\begin{prop}
		For any $k,n\in \N$, the families $\{(x,y)\mapsto x,(x,y)\mapsto y,(x,y)\mapsto x\cdot y^{n},(x,y)\mapsto x+y,\cdots,(x,y)\mapsto x+ky\}$ and $\{(x,y)\mapsto x,(x,y)\mapsto y,(x,y)\mapsto x\cdot y^{-n},(x,y)\mapsto x+y,\cdots,(x,y)\mapsto x+ky\}$ are Ramsey on $\Q$.
	\end{prop}
	 Likely, for general polynomials with zero constant term, if one can build the upper Banach density version of Proposition \ref{prop3} , it is possible to confirm the color of $y$.
 \section*{Acknowledgement}
 The author is supported by NNSF of China (11971455, 12031019, 12090012). The author's thanks go to Professor Song Shao for his useful suggestions.

 \bibliographystyle{plain}
 \bibliography{ref}
 
 \end{document}